\documentclass[11pt,twoside]{article}

\usepackage{amssymb}
\usepackage{epsfig,enumerate,amsmath,amsfonts,amssymb}
\usepackage{indentfirst}
\usepackage{pstricks}
\usepackage{setspace}
\usepackage{latexsym}
\usepackage{amsthm}
\usepackage{amsmath}

\newtheorem{defi}{Definition}[section]
\newtheorem{thm}{Theorem}[section]

\newtheorem{lem}{Lemma}[section]
\newtheorem{prop}{Proposition}[section]
\usepackage[all]{xy}
\catcode`\@=11
\def\theequation{\@arabic{\c@section}.\@arabic{\c@equation}}
\catcode`\@=12

\begin{document}

\title
{Existence and multiplicity results for fractional $p$-Kirchhoff equation with sign changing nonlinearities}
\author{
{\bf  Pawan Kumar Mishra\footnote{email: pawanmishra31284@gmail.com}} and {\bf  K. Sreenadh\footnote{e-mail: sreenadh@gmail.com}}\\
{\small Department of Mathematics}, \\{\small Indian Institute of Technology Delhi}\\
{\small Hauz Khaz}, {\small New Delhi-16, India.}\\
}

\date{}

\maketitle
\begin{abstract}
\noindent In this paper, we show the existence and multiplicity of nontrivial, non-negative solutions of the  fractional $p$-Kirchhoff problem \\
\begin{equation*}
\begin{array}{rllll}
M\left(\displaystyle\int_{\mathbb{R}^{2n}}\frac{|u(x)-u(y)|^p}{\left|x-y\right|^{n+ps}}dx\,dy\right)(-\Delta)^{s}_p u
&=\lambda f(x)|u|^{q-2}u+ g(x)\left|u\right|^{r-2}u\;  \text{in } \Omega,\\
u&=0 \;\mbox{in } \mathbb{R}^{n}\setminus \Omega,
\end{array}
\end{equation*}
where $(-\Delta)^{s}_p$ is the fractional $p$-Laplace operator,
$\Omega$ is a bounded domain in $\mathbb{R}^n$ with smooth boundary, $f \in L^{\frac{r}{r-q}}(\Omega)$ and $g\in L^\infty(\Omega)$  are sign changing, $M$ is continuous function, $ps<n<2ps$ and $1<q<p<r\leq p_s^*=\frac{np}{n-ps}$.
\end{abstract}
\section {Introduction}
\noindent In this work, we study the existence and multiplicity of solutions for the following $p$-Kirchhoff equation
\begin{equation*}
(P_\lambda)\;
\left\{\begin{array}{rllll}
M\left(\displaystyle\int_{\mathbb{R}^{2n}}\frac{|u(x)-u(y)|^p}{\left|x-y\right|^{n+ps}}dx\,dy\right)(-\Delta)^{s}_p u
&=\lambda f(x)|u|^{q-2}u+ g(x)\left|u\right|^{r-2}u,\;\; u\ge 0, \; u\not\equiv 0\;  \text{in } \Omega,\\
u&=0 \;\mbox{in } \mathbb{R}^{n}\setminus \Omega,
\end{array}
\right.
\end{equation*}
where $(-\Delta)^{s}_p$ is the fractional $p$-Laplace operator defined as
\begin{equation*}
{(-\Delta)^{s}_pu(x)}=- 2\int_{\mathbb R^n}\frac{|u(y)-u(x)|^{p-2}(u(y)-u(x))}{|x-y|^{n+ps}}
dy
\end{equation*}
where $M(t)=a+bt,\;a,b>0$,\;$p\geq2$, $1<q<p<r\leq p^*_s$, $ps < n < 2ps$ with $s\in(0,1)$, $\lambda$ is a positive parameter, $\Omega\subset\mathbb{R}^{n}$ is a bounded domain with smooth boundary and $f$ and $g$ satisfy the following assumptions,\\
\textbf{(f1)} $f\in L^{\frac{r}{r-q}}(\Omega)$ and g $\in L^\infty(\Omega)$ and
\textbf{(f2)} $f^+(x)=\max \{f(x), 0\}\not \equiv 0$ in $\bar{\Omega}$ and $g^+(x)=\max \{g(x), 0\}\not \equiv 0$ in $\bar{\Omega}$ (f and g are possibly sign changing on $\bar{\Omega}$). Fractional $p$-Laplacian operator is quasilinear generalization of fractional Laplace operator $(-\Delta)^s_2$.

\noindent In this work, we would like to study the existence and multiplicity of solutions for the problem $(P_\lambda)$ with sign changing convex-concave type nonlinearity. {A subcritical problem of this type is studied in \cite{SS2} with $M=1$. The presence of $M(\|u\|_{X_0}^p)$ in the weak formulation requires the compactness of Palais-Smale sequences to obtain the solution.  In order to show the existence of solution, we need to show the strong convergence of weak limit of Palais-Smale sequence which is not the case for $M\equiv 1$. We prove the strong convergence of Palais-Smale sequences using the concentration compactness \cite{pal}. The growth of Kirchhoff term makes the study of fibering analysis more complicated.}
Starting from the  pioneering work of {Ambrosetti-Brezis-Cerami \cite{amb1994}}, there are many works on multiplicity results with convex and concave type nonlinearities. These results were generalized to $p$-Laplacian type equations by many authors {see \cite{CAO}, \cite{VER} and references therein}.

\noindent During the last one decade, several authors used the Nehari manifold and associated fiber maps to solve semilinear and quasilinear elliptic
problems when the nonlinearity changes sign. In \cite{ken2007, wu2008}, authors studied the existence and multiplicity of solutions for subcritical and critical nonlinearities via Nehari manifold method for semilinear equation and in \cite{dpp}, these results are studied for $p$-Laplacian equation. In \cite{cyuyc}, authors considered the following Kirchhoff equation:
\begin{equation*}
M(\int_{\Omega} |\nabla u|^2 dx) \Delta u = \lambda f(x)|u|^{q-2} u+g(x) |u|^{p-2}u, \; \text{in} \; \Omega, \;\; u=0\; \text{on} \; \partial \Omega,
\end{equation*}
where $1<q<2<p<2^*$ and $f(x)$ and $g(x)$ are sign changing continuous functions. They studied the existence and multiplicity results by studying the structure of Nehari manifold in different cases $(i)\; p<4\; (ii)\; p=4 $ and $(iii)\; p>4$. The above problem is called non-local because of the presence of integral over $\Omega$. It means that $(P_\lambda)$ is no more a point wise identity. Non-local operators, naturally arise in several physical and biological models such as continuum mechanics, phase transition phenomena, population dynamics and game theory (see \cite{caf} and references therein).

\noindent  The fractional Laplacian  operator has been a classical topic in Fourier analysis and nonlinear partial differential equations for a long time.  Fractional operators are also involved in financial mathematics, where Levy processes with jumps appear in modeling the asset prices (see \cite{app}).
The critical exponent problems for square root of Laplacian are studied in \cite {capella}, \cite{tan}. The Brezis-Nirenberg type results for fractional Laplacian $(-\Delta)^s, (0<s<1)$ are studied in  \cite {MN}, \cite {sv3}, \cite{sv1}. In \cite {yu}, authors studied the Nehari manifold for square root of Laplacian by considering the harmonic extension of solutions in the half cylinder $\Omega\times (0, \infty)$ vanishing on the boundary {$\partial\Omega\times[0, \infty)$}. The idea of these harmonic extensions was initially introduced and studied in the beautiful work of Caffarelli and Silvestre\cite{CS}. In \cite {SS2}, existence of solutions for $p$-fractional Laplacian equations with sign changing nonlinearities are studied  via the method of Nehari manifold.
\noindent The Kirchhoff equations with Laplacian and $p$-Laplacian operators are widely studied by many authors using variational methods ( see \cite{acf}, \cite{col}, \cite {gf}, \cite{ham2011}and references therein ).

In \cite{afev}, authors considered the following Kirchhoff equation with fractional Laplacian:
\begin{equation*}
M\left(\displaystyle\int_{\mathbb{R}^{2n}}\frac{|u(x)-u(y)|^2}{\left|x-y\right|^{n+2s}}dx\,dy\right)(-\Delta)^{s}_{2} u
=\lambda f(x,u)+\left|u\right|^{2_{s}^{*}-2}u\;  \text{in } \Omega,\quad
u=0 \;\mbox{in } \mathbb{R}^{n}\setminus \Omega,
\end{equation*}
where $\lambda$ is positive parameter and $f(x,t)$ is a subcritical term having  superlinear growth at $t=0$.
They studied the existence of mountain pass solutions using  the concentration compactness principle of Lions for fractional Sobolev spaces \cite{pal}.
To the best of our knowledge, there are no works in the literature dealing with Nehari manifold and $p$-fractional equations with sign changing nonlinearities to study the existence and multiplicity results.

\noindent In this work, we prove the existence of multiple non-negative solutions for the Kirchhoff elliptic equation with $p$-fractional operator and sign changing nonlinearity by studying the nature of Nehari manifold with respect to the parameter
$\lambda$ and fiber maps associated with the Euler functional. In the subcritical case, we show the existence of local minimizers of the associated functional on nonempty decompositions of Nehari manifold. We show the multiplicity result by extracting Palais-Smale sequences in the Nehari manifold. The results obtained here are somehow expected but we show how the results arise out of nature of Nehari manifold.  For this we adopt the approach in \cite{cyuyc}. We also prove an existence result for critical case  using concentration compactness Lemma for suitable range  of $\lambda$.

\noindent The natural space to look for  solutions is the fractional Sobolev space $W^{s,p}_{0}(\Omega)$. In order to study $(P_\lambda)$, it is important to encode the `boundary condition' $u=0$ in $\mathbb{R}^n\setminus\Omega$ in the weak formulation. It was observed in \cite{sv2, sv3} that the interaction between $\Omega$ and $\mathbb{R}^n\setminus\Omega$, which gives positive contribution in the norm $\left\|u\right\|_{W^{s,p}(\mathbb{R}^n)}$. They introduced new function spaces to study the variational functionals related to fractional Laplacian. Subsequently, in \cite{SS2}, authors studied the $p$-fractional equations on the function spaces defined as:\\
For $1<p<\infty,\; p\neq 2$, define the space $X$ as
\begin{equation*}
X=\left\{u \mid u:\mathbb{R}^n\rightarrow \mathbb{R}\; \text{is measurable,}\; u\mid_\Omega\in L^p(\Omega), \; \left(\frac{u(x)-u(y)}{|x-y|^{\frac{n}{p}+s}}\right)\in L^p(Q) \right\},
\end{equation*}
where $Q=\mathbb{R}^{2n}\setminus\left({\mathcal C}\Omega\times{\mathcal C}\Omega\right)$ \;and \;$C\Omega=\mathbb{R}^n\setminus \Omega$.
Then the space $X$ endowed with the norm, defined as
\begin{equation}\label{norma}
\left\|u\right\|_{X}=\Big(\|u\|_{L^p(\Omega)}+\int_Q \frac{|u(x)-u(y)|^p}{|x-y|^{n+ps}}dx\,dy\Big)^{1/p}
\end{equation}
is a reflexive Banach space. It is immediate to observe that $C_{c}^{1}(\Omega)\subseteq X$. The function space $X_0$ denotes the closure of $C^{\infty}_{0}(\Omega)$ in $X$. The space $X_0$ is a Banach space which can be endowed with the norm, defined as
\begin{equation}\label{normaz}
\|u\|_{X_0}=\Big(\int_Q \frac{|u(x)-u(y)|^p}{|x-y|^{n+ps}}dx\,dy\Big)^{1/p}\,.
\end{equation}
Note that in equations \eqref{norma} and \eqref{normaz}, the integrals can be extended to $\mathbb{R}^{2n}$, since  $u=0$ a.e. in {$\mathbb{R}^n \setminus\Omega$}. The Energy functional associated to the problem  $(P_\lambda)$ is given by
\begin{equation}\label{enj}
\mathcal{J}_\lambda (u)=\frac{1}{p}\widehat{M}(\|u\|_{X_0}^p)-\frac{\lambda}{q}\int_\Omega f(x)|u|^qdx-\frac{1}{r}\int_\Omega g(x)|u|^rdx
\end{equation}
where {$\widehat{M}(t)=\int_0^t M(s) $ is the primitive of $M$.}
\begin{defi}
A function $u\in X_0$ is called  weak solution of $(P_\lambda)$ if $u$ satisfies
\begin{align}
\displaystyle M(\left\|u\right\|^p_{X_0})\int_{\mathbb{R}^{2n}} &\left|u(x)-u(y)\right|^{p-2}(u(x)-u(y))(\varphi(x)-\varphi(y))|x-y|^{-n-ps
} dx\,dy\notag\\
&= \displaystyle\lambda\int_\Omega f(x)|u|^{q-2}u\varphi(x)\,dx+\int_\Omega g(x){\left|u\right|^{r-2}}u(x)\varphi(x)dx \,\,\,\,\,\,\,\textrm{for all}\;\;\varphi \in X_0.
\label{wf}
\end{align}
\end{defi}
\noindent The Nehari set $\mathcal{N_\lambda}$ associated to the problem $(P_\lambda)$ is defined as
\begin{equation}\label{bsic}
\mathcal{N_\lambda}=\{{u\in X_0\setminus\{0\}}: \langle \mathcal{J_\lambda}^\prime(u), u\rangle =0\},\end{equation}
where $\langle,\rangle$ is the duality between $X_0$ and its dual space. First we show that, the set $\mathcal{N}_\lambda$ is a manifold for small $\lambda$.
\begin{thm}\label{th1.1}
There exists $\lambda_0>0$ such that $\mathcal{N_\lambda}$ is a $C^1$ manifold for all $\lambda\in(0,\lambda_0)$.
\end{thm}
\noindent {We show the following existence and multiplicity theorems} for all $\lambda\in (0,\lambda_0)$ in the subcritical case. We remark that these results are sharp in the sense that for $\lambda>\lambda_0$, it is not clear if $\mathcal{N_\lambda}$ is a manifold.
\begin{thm}\label{th1}
Let $r\in (2p, p^*$). Then, $(P_\lambda)$ has at least two  solutions for $\lambda \in(0, \lambda_0)$.
\end{thm}
\noindent For the second theorem, we consider the following minimization problem
\begin{equation}\label{minmm}
\Lambda=\inf\left\{ \|u\|_{X_0}^{2p}\;:\;u\in X_0, \int_\Omega g(x)|u|^{2p} dx=1\right\}.
\end{equation}
Using direct methods of calculus of variations, we can show that $\Lambda >0$ and is achieved by $u_\Lambda\geq0$ in $\Omega$.
\begin{thm}\label{th22} Let $r=2p$ and let $\Lambda$ be as in \eqref{minmm}. Then
\begin{enumerate}[(i)]
\item $(P_\lambda)$ has at least one solution for all $\lambda>0.$
\item For $b<\frac{1}{\Lambda}$,\; there exists a $\lambda ^0>0$ such that $(P_\lambda)$ has at least two solutions for $\lambda \in(0, \lambda^0)$.
\end{enumerate}
\end{thm}
\noindent For our next result, we use the following version of Lemma 2 of \cite{col}.
\begin{lem}\label{two}
Let u be a non-negative solution of
\begin{equation}\label{lpp}
M(\|u\|_{X_0}^p)(-\Delta)^{s}_p u
=h_\lambda(x,u)\  \text{in }\ \Omega,\;\; u=0 \;\mbox{in } \mathbb{R}^{n}\setminus \Omega,
\end{equation}
where $h_\lambda\in C(\Omega\times\mathbb{R})$ satisfies
\begin{equation}\label{grth}
h_\lambda(x,t)\leq \lambda C_0|t|^q+C_1|t|^r \;\;\textrm{for all}\; x\in \Omega ,t\in \mathbb{R}
\end{equation}
 and $C_0\geq 0,\; C_1>0,\;\lambda>0,\; 0<q<r,\; 1<r<p_s^*-1$.
Then there exists $C_*>0$, independent of $M$, such that
\begin{equation} \label{1.4new}
\|u\|_{X_0}^p<\max\{M(\|u\|_{X_0}^p)^{\frac{q-r+2}{r-1}}, M(\|u\|_{X_0}^p)^{\frac{2}{r-1}}\}(\lambda C_0C_*^{q+1}+C_1C_*^{r+1})|\Omega|.
\end{equation}
\end{lem}
\begin{proof}
Let $u$ be a non-negative solution of \eqref{lpp}. Then $v=\frac{u}{M(\|u\|_{X_0}^p)^{\frac{1}{r-1}}}$ is a {non-negative} solution of
\begin{equation*}
(-\Delta)^{s}_p v
=g_\lambda(x,v)\;  \text{in } \Omega,\;v=0 \;\mbox{in } \mathbb{R}^{n}\setminus \Omega.
\end{equation*}
By \eqref{grth}, and the fact that $M(s)\ge a$ for all $s$, we see that
\begin{align*}
|g_\lambda(x,s)|&={\frac{|h_\lambda\left(x, M(\|u\|_{X_0}^p)^{\frac{1}{r-1}}s\right)|}{M(\|u\|_{X_0}^p)^{\frac{p-1}{r-1}+1}}}\\
&\leq M^{\frac{q-p+1}{r-1}-1}\lambda C_0|s|^q + M^{\frac{r-p+1}{r-1}-1}C_1|s|^r\\
&\leq {a^{\frac{r-q+p-2}{r-1}}\lambda C_0|s|^q + a^{\frac{p-2}{r-1}}C_1|s|^r}
\leq \lambda C_2|s|^q+C_3|s|^r,
\end{align*}
for some positive constants $C_2$ and $C_3.$ Now for $r<2p$, we have $1+\frac{q}{p}>\frac{r}{p}+\frac{r}{p_s^*}$. So by Theorem 3.1 of \cite{ls}, \;$\|v\|_{\infty}\leq C_* $, for some $C_*>0$ (independent of $M$) . Therefore
\begin{equation*}
\|u\|_\infty\leq M(\|u\|_{X_0}^p)^{\frac{1}{r-1}}C_*.
\end{equation*}
 Since $u$ solves  \eqref{lpp}, multiplying \eqref{lpp} by $u$  and integrating by parts, we get
\begin{align*}
\|u\|_{X_0}^p&=M(\|u\|_{X_0}^p)^{-1}\int_\Omega |h_\lambda(x,u)u |dx\\
&\leq M(\|u\|_{X_0}^p)^{-1}\left(\lambda C_0\|u\|_\infty^{q+1}+C_1\|u\|_\infty^{r+1}\right)|\Omega|\\
&\leq\max\{M(\|u\|_{X_0}^p)^{\frac{q-r+2}{r-1}}, M(\|u\|_{X_0}^p)^{\frac{2}{r-1}}\}\left(\lambda C_0C_*^{q+1}+C_1C_*^{r+1}\right)|\Omega|.
\end{align*}
\end{proof}
\noindent Now for  the case $r<2p$,  we define
$L(\lambda)=\left(\lambda C_0C_*^{q+1}+C_1C_*^{r+1}\right)|\Omega|$ and
\begin{equation}
\label{newint}
\widehat{A}=\max_{k\in I}\left\{M(k)^{\frac{q-r+2}{r-1}}, M(k)^{\frac{2}{r-1}}\right\}, \; \text{where} \;  I= \left(\frac{a(r-p)}{rb}, \frac{a(r-p)}{pb}\right).
\end{equation}
\noindent Then we have the following theorem.
\begin{thm}\label{th33} Let $r<2p$ and let $\hat{A}$ be defined in \eqref{newint}.  Then
\begin{enumerate}[(i)]
\item $(P_\lambda)$ has at least one solution for each $\lambda>0$.
\item  For any $\theta>0$, and $\displaystyle 0<b<\frac{a(r-p)}{r\widehat{A}L(\theta)}$, there exists $\hat{\lambda}_0\in(0,\theta]$ such that $(P_\lambda)$ has at least two solutions for each $0<\lambda<\hat{\lambda}_0$.
 \end{enumerate}
 \end{thm}
\noindent Finally, in the case of critical nonlinearity, we have the following {theorem.}
\begin{thm}\label{th2}
Let $r=p_{s}^{*}$ and $g(x)\equiv 1$. Then there exists a $\lambda_{00}>0$ such that $(P_\lambda)$ admits at least one solution for all $\lambda \in(0, \lambda_{00})$.
\end{thm}
\noindent The paper is organized as follows: In section 2, we introduce Nehari manifold for $(P_\lambda)$ and
show that the associated Euler functional is bounded below on this manifold. Section 3
contains the existence and multiplicity results in the subcritical case. In section 4, we have the existence of
nontrivial solution in the critical case.
\section{Nehari manifold for $(P_{\lambda})$}
\setcounter{equation}{0}
\noindent In this section we describe the nature of Nehari manifold corresponding to the problem $(P_\lambda)$.  In the case $r\ge 2p$, the functional $\mathcal{J}_\lambda$ is not bounded below on $X_0$ since $\widehat{M}(t)\sim t^{2p}$ as $t\rightarrow \infty$. We will show that it is bounded on some suitable subset of $X_0$ and on minimizing $\mathcal{J}_\lambda$ on these subsets, we get the solutions for problem $(P_\lambda)$.
From the definition of $\mathcal{N_\lambda}$,  $u\in \mathcal N_{\lambda}$ if and only if
\begin{equation}\label{eq2}
M(\|u\|_{X_0}^p)\int_{Q} |u(x)-u(y)|^p|x-y|^{-n-ps} dxdy - \lambda \int_{\Omega} f(x)|u|^q dx-
\int_{\Omega}g(x)|u|^{r}dx =0 .
\end{equation}
We note that $\mathcal N_{\lambda}$ contains every non zero solution of $(P_\lambda)$. Now as we know that the Nehari manifold is closely
related to the behavior of the functions $\phi_u: \mathbb R^+\rightarrow \mathbb R$
defined as $\phi_{u}(t)=\mathcal{J}_{\lambda}(tu)$. Such maps are called fiber
maps and were introduced by Drabek and Pohozaev in \cite{dpp}. For
$u\in X_0$, we have
\begin{align}\label{phde}
\phi_{u}(t) &=\frac{1}{p}\widehat{M}(t^p\|u\|_{X_0}^p)-\frac{\lambda}{q}t^q\int_\Omega f(x)|u|^qdx-\frac{1}{r}t^r\int_\Omega g(x)|u|^rdx,\\\label{phid1}
\phi_{u}^{\prime}(1) &=M(\|u\|_{X_0}^p)\|u\|_{X_0}^p-\lambda \int_\Omega f(x)|u|^qdx-\int_\Omega g(x)|u|^rdx,\\\nonumber
\phi_{u}^{\prime\prime}(1) &= (p-1)M(\|u\|_{X_0}^p)\|u\|_{X_0}^p+M'(\|u\|_{X_0}^p)p\|u\|_{X_0}^{2p}\\\label{phid2}
&\quad \quad-\lambda(q-1) \int_\Omega f(x)|u|^qdx-(r-1)\int_\Omega g(x)|u|^rdx.
\end{align}
Then it is easy to see that $u\in \mathcal N_{\lambda}$ if and only if
$\phi_{u}^{\prime}(1)=0$. Thus it is natural to split
$\mathcal {N}_{\lambda}$ into three parts corresponding to local minima,
local maxima and points of inflection. For this, we set
\begin{align*}
\mathcal N_{\lambda}^{\pm}&:= \left\{u\in \mathcal N_{\lambda}:
\phi_{u}^{\prime\prime}(1)
\gtrless0\right\} =\left\{tu\in X_0 : \phi_{u}^{\prime}(t)=0,\; \phi_{u}^{''}(t)\gtrless  0\right\},\\
\mathcal N_{\lambda}^{0}&:= \left\{u\in \mathcal N_{\lambda}:
\phi_{u}^{\prime\prime}(1) = 0\right\}=\left\{tu\in X_{0} :
\phi_{u}^{\prime}(t)=0,\; \phi_{u}^{''}(t)= 0\right\}.
\end{align*}
We define $H^\pm=\{u\in X_0\,:\, \int_\Omega f(x)|u|^q dx\gtrless0\}$,\; $G^\pm=\{u\in X_0\,:\, \int_\Omega g(x)|u|^r dx\gtrless0\}$,\; $H^0=\{u\in X_0\,:\,\int_\Omega f(x)|u|^q dx=0\}$ and $G^0=\{u\in X_0\,:\,\int_\Omega g(x)|u|^r dx=0\}$.
Define $\psi_u:\mathbb{R}^+\rightarrow \mathbb{R}$ as
\begin{equation}\label{siut}
\psi_u(t)=at^{p-q}\|u\|_{X_0}^p+bt^{2p-q}\|u\|_{X_0}^{2p}-t^{r-q}\int_\Omega g(x)|u|^r dx.
\end{equation}
Then
\begin{equation}\label{siutd}
 \psi_u^{\prime}(t)=a(p-q)t^{p-q-1}\|u\|_{X_0}^p+b(2p-q)t^{2p-q-1}\|u\|_{X_0}^{2p}-(r-q)t^{r-q-1}\int_\Omega g(x)|u|^r dx.
\end{equation}
We also note that $\phi_{tu}$ and $\psi_u$ satisfies
\begin{equation}\label{psr}
\phi_{tu}^{\prime\prime}(1)=t^{-q-1}\psi_u^{\prime}(t).
\end{equation}
\begin{lem}\label{lmcp}
If $u$ is a minimizer of $\mathcal{J}_{\lambda}$ on $\mathcal{N}_{\lambda}$ such that $u \notin \mathcal{N}_{\lambda}^{0}.$ Then $u$ is a critical point for $\mathcal{J}_{\lambda}.$
\end{lem}
\begin{proof}
{The details of the proof can be found in \cite{dpp}.}
\end{proof}
\noindent Define $\theta_{\lambda} := \inf\{\mathcal{J}_{\lambda}(u)| u \in \mathcal{N}_{\lambda}\}$ and $\theta_{\lambda}^\pm := \inf\{\mathcal{J}_{\lambda}(u)| u \in \mathcal{N}_{\lambda}^\pm\}$. Then
\begin{lem}\label{le44}
$\mathcal{J}_{\lambda}$ is coercive and bounded below on  $\mathcal{N}_{\lambda}$.
\end{lem}
\begin{proof}
For $u \in \mathcal{N}_{\lambda},$  using H$\ddot{\textrm{o}}\textrm{lder}'$s inequality, we have
\begin{eqnarray*}
\mathcal{J}_{\lambda}(u) &=& \left(\frac{1}{p}-\frac{1}{r}\right)a\|u\|_{X_0}^{p}+ \left(\frac{1}{2p}-\frac{1}{r}\right)b\|u\|_{X_0}^{2p}-\left(\frac{1}{q}-\frac{1}{r}\right)\int_\Omega f(x)|u|^{q}dx,\\
 &\geq& \left(\frac{1}{p}-\frac{1}{r}\right)a\|u\|_{X_0}^{p}+ \left(\frac{1}{2p}-\frac{1}{r}\right)b\|u\|_{X_0}^{2p}-\lambda \left(\frac{1}{q}-\frac{1}{r}\right)l^{\frac{r-q}{r}} \left(\int_\Omega |u|^{r}dx\right)^{q/r},\\
&\geq& \left(\frac{1}{p}-\frac{1}{r}\right)a\|u\|_{X_0}^{p}+ \left(\frac{1}{2p}-\frac{1}{r}\right)b\|u\|_{X_0}^{2p}-\lambda \left(\frac{1}{q}-\frac{1}{r}\right)l^{\frac{r-q}{r}} S_r^{-q}\|u\|_{X_0}^q\\
\end{eqnarray*}
Thus $\mathcal{J}_\lambda$ is coercive and bounded below in $\mathcal{N}_\lambda$ for $r>2p$.
\end{proof}
\begin{lem}\label{2s1}
There exists $\lambda_{0} > 0$ such that $\mathcal{N}_{\lambda}^{0}(\Omega) = \emptyset, \;\forall \;\lambda \in (0, \lambda_{0})$
\end{lem}
\begin{proof}
We have following two cases.
\begin{enumerate}[]
  \item \textbf{Case 1:}  $u \in \mathcal{N}_{\lambda}(\Omega)$ and $ \int_{\Omega}f(x) |u|^{q}dx = 0.$\\
  From equation \eqref{phid1}, we have,\; $a \|u\|_{X_0}^{p} + b\|u\|_{X_0}^{2p}- \int_\Omega g(x)|u|^r dx = 0$. Now,
  \begin{eqnarray*}
    pa \|u\|_{X_0}^{p} + 2pb\|u\|_{X_0}^{2p}- r\int_\Omega g(x)|u|^r dx  &=& pa \|u\|_{X_0}^{p} + 2pb\|u\|_{X_0}^{2p}- r (a \|u\|^{p} + b\|u\|_{X_0}^{2p})\\
     &=&  {(p-r)a\|u\|_{X_0}^{p}+(2p-r)b\|u\|_{X_0}^{2p}<0.}
  \end{eqnarray*}
  {which implies $u \notin \mathcal{N}_{\lambda}^{0}(\Omega)$}
  \item \textbf{Case 2:} $u \in \mathcal{N}_{\lambda}(\Omega)$ and $ \int_{\Omega}f(x) |u|^{q}dx \neq 0.$
  \vspace*{.1 cm}\\
  Suppose $u \in \mathcal{N}_{\lambda}^{0}(\Omega)$ . Then from equations \eqref{phid1} and \eqref{phid2}, we have
    \begin{eqnarray}\label{2.11}
    (p-q)a \|u\|_{X_0}^{p} + (2p-q)b\|u\|_{X_0}^{2p} &=& (r-q)\int_\Omega g(x)|u|^r dx,\\\label{2.12}
    (r-p) a \|u\|_{X_0}^{p} + (r-2p) b\|u\|_{X_0}^{2p} &=& (r-q) \lambda \int_\Omega f(x)|u|^q dx.
  \end{eqnarray}
  Define $E_{\lambda}: N_{\lambda}(\Omega) \rightarrow R$ as
  \begin{equation*}
    E_{\lambda}(u) = \frac{(r-p) a \|u\|_{X_0}^{p} + (r-2p) b\|u\|_{X_0}^{2p}}{(r-q)} - \lambda\int_\Omega f(x)|u|^{q} dx,
  \end{equation*}
 { then, from equation \eqref{2.12}, $E_{\lambda}(u) = 0, \; \forall u\; \in \mathcal{N}_{\lambda}^{0}(\Omega).$  Also,}
  \begin{eqnarray*}
  E_{\lambda}(u) & \geq & \left(\frac{r-p}{r-q}\right)a\|u\|_{X_0}^p - \lambda \int_\Omega f(x)|u|^q dx\\
  & \geq & \left(\frac{r-p}{r-q}\right)a\|u\|_{X_0}^p - \lambda {\|f\|_{L^{\frac{r}{r-q}}}(\Omega)}\|u\|_{X_0}^qS_r^{-q},\\
  & \geq & \|u\|_{X_0}^q\left[\left(\frac{r-p}{r-q}\right)a\|u\|_{X_0}^{(p-q)} - \lambda {\|f\|_{L^{\frac{r}{r-q}}}(\Omega)}S_r^{-q}\right],\\
  \end{eqnarray*}
  Now, from equation \eqref{2.11}, we get
  \begin{equation}\label{2.13}
  \|u\| \geq \left[\left(\frac{p-q}{r-q}\right)\frac{aS_r^{r}}{\|g\|_\infty}\right]^{\frac{1}{r-p}}.
  \end{equation}
  Using equation (\ref{2.13}), we get
   \begin{equation*}
    E_{\lambda}(u) \geq \|u\|_{X_0}^q \left[(r-p) \left(\frac{a}{r-q}\right)^{\frac{r-q}{r-p}}\left(\frac{(p-q)S_r^r}{\|g\|_\infty}\right)^{\frac{p-q}{r-p}}- \lambda {\|f\|_{L^{\frac{r}{r-q}}}(\Omega)}S_r^{-q}\right].
  \end{equation*}
  This implies that there exists $\lambda_1>0$ such that for $\lambda\in (0, \lambda_1),\;E_{\lambda}(u)>0, \; \forall\; u \in \mathcal{N}_{\lambda}^{0}(\Omega),$
  which is contradiction. Hence, $\mathcal{N}_{\lambda}^{0}(\Omega) = \phi.$
\end{enumerate}
\end{proof}

\noindent {\bf{Proof of { Theorem \ref{th1.1}}:}} Now the proof of Theorem 1.1 follows from the implicit function theorem and Lemma \ref{2s1}{.}

\section{Existence and multiplicity results in the subcritical case}
\setcounter{equation}{0}
\noindent In this section, we prove the existence and multiplicity results for the case $r<p_s^*$.
  We need the following Lemmas.
\begin{lem}\label{pscn}
Every Palaise-Smale sequence of $\mathcal{J}_\lambda$ has a convergent subsequence. That is,
if $\{u_k\}\subset X_0$ satisfies
\begin{equation}\label{pscn1}
\mathcal{J}_\lambda(u_k)=c+o_k(1) \;\text{and}\; \mathcal{J}'_\lambda(u_k)=o_k(1)\; \textrm{in}\; X_0^{\prime},
\end{equation}
 then $\{u_k\}$ has a convergent subsequence.
\end{lem}
\begin{proof}
Let $\{u_{k}\}$ be a sequence satisfying equation \eqref{pscn1}. Then it is easy to verify that $\{u_{k}\}$ is bounded in $X_{0}$. So upto subsequence $u_{k} \rightharpoonup u_{\lambda}$ weakly in $\mathrm{X}_{0}$, $u_{k} \rightarrow u_{\lambda}$ strongly in $\mathrm{L}^{q}(\Omega), 1 \le q < p_s^{*}$ and $u_{k}(x) \rightarrow u_{\lambda}(x)$ a.e. in $\Omega$. Now, by compactness of the imbedding $X_0\hookrightarrow L^{\alpha}(\Omega)$ for all $\alpha<p_{s}^{*}$, we have
\begin{equation*}
    \int_\Omega{f(x)|u_{k}|^{q}dx} \rightarrow \int_\Omega{f(x)|u_{\lambda}|^{q}dx}\;\;\;\mathrm{as}\;\;\; k \rightarrow \infty.
\end{equation*}
{Therefore, $\langle\mathcal{J}_{\lambda}^{'}(u_k)-\mathcal{J}_{\lambda}^{'}(u_\lambda), (u_{k}-u_{\lambda})\rangle \rightarrow 0$ as $k\rightarrow \infty$.
Now using the inequality $|a-b|^{l} \leq 2^{l-2}(|a|^{l-2}a-|b|^{l-2}b)(a-b) \;\;\;\textrm{for}\; a, b \in \mathbb{R}^{n},\; l \geq 2,$ and  $M(s) \geq a$ we get, $\|u_{k}-u_{\lambda}\|_{X_0} \rightarrow 0$ as $k \rightarrow \infty.$}
\end{proof}
\begin{lem}\label{L37}
(i) For every $u \in H^+\cap G^+ $, there is a unique $t_{\max}=t_{\max}(u)>0$
and unique $t^+(u)<t_{\max}<t^-(u)$ such that $t^+u\in \mathcal{N}_\lambda^+, t^-u\in \mathcal{N}_\lambda$ and $\mathcal{J}_\lambda(t^+u)=\displaystyle\min_{0\leq t\leq t^-} \mathcal{J}_\lambda(tu)$,
$\mathcal{J}_{\lambda}(t^-u) = \displaystyle \max_{t\geq t_{\max}} \mathcal{J}_{\lambda}(tu)$.\\
(ii) For $u \in H^+\cap G^- $, there exists a unique $t^*>0$ such that $t^*u\in \mathcal{N}_\lambda^+$ and $\mathcal{J}_\lambda(t^*u)=\displaystyle\min_{ t\geq 0} \mathcal{J}_\lambda(tu)$
\end{lem}
\begin{proof}
Let $u \in H^{+}\cap G^+.$ Then from equation \eqref{siutd}, we note that $\psi_{u}(t) \rightarrow - \infty$ as $t \rightarrow \infty$. From equation \eqref{siutd}, it is easy to see that
$\displaystyle\lim_{t\rightarrow 0^{+}}\psi^{'}_{u}(t)>0$ and $\displaystyle\lim_{t\rightarrow \infty}\psi^{'}_{u}(t)<0$. So there exists a unique $t_{\max} = t_{\max}(u)>0$ such that $\psi_{u}(t)$ is increasing on $(0, t_{\max})$, decreasing on $(t_{\max}, \infty)$ and $\psi^{'}_{u}(t_{\max}) = 0.$
\begin{equation*}
    \psi_{u}(t_{\max}) = t^{-q}_{\max}\left(a\;t^{p}_{\max}\|u\|_{X_0}^{p} + b\;t^{2p}_{\max}\|u\|_{X_0}^{2p}- t_{\max}^{r}\int_\Omega g(x){|u|^r}\right)
\end{equation*}
where $t_{\max}$ is the root of
\begin{equation}\label{2secnew}
    a(p-q)t_{\max}^{p}\|u\|_{X_0}^{p} + b(2p-q)t_{\max}^{2p}\|u\|_{X_0}^{2p}-(r-q)t^{r}_{\max}\int_\Omega g(x){|u|^{r}} = 0.
\end{equation}
Now from equation \eqref{2secnew}, we get
\begin{equation}\label{eneq}
t_{\max} \geq \frac{1}{\|u\|_{X_0}} \left[\frac{a(p-q)S_r^r}{(r-q)\|g\|_\infty}\right]^{\frac{1}{r-p}}:=T_0
\end{equation}
Using inequality {\eqref{eneq}}, we have
\begin{eqnarray*}
    \psi_{u}(t_{\max}) & \geq & \psi_u(T_0)\geq a T_0^{p-q}\|u\|_{X_0}^p-T_0^{r-q}\int_\Omega g(x) |u|^{r}dx \\
    & \ge  & \|u\|_{X_0}^q (r-p)\left(\frac{a}{r-q}\right)^{\frac{r-q}{r-p}}\left(\frac{(p-q)S_r^r}{\|g\|_\infty}\right)^{\frac{p-q}{r-p}}>0
\end{eqnarray*}
Hence if
$\lambda<\lambda_2=\left(\frac{(r-p)S_r^q}{{\|f\|_{L^{\frac{r}{r-q}}}(\Omega)}}\right)
\left(\frac{a}{r-q}\right)^{\frac{r-q}{r-p}}\left(\frac{(p-q)S_r^r}{\|g\|_\infty}\right)^{\frac{p-q}{r-p}}$
, then there exists unique $t^+ = t^+(u) {<} t_{\max}$ and $t^- = t^-(u) > t_{\max},$ such that $\psi_{u}(t^+) = \lambda \int_{\Omega}{f(x)|u|^{q}} = \psi_{u}(t^-)$. That is,  $t^+u, t^-u \in \mathcal{N}_{\lambda}.$ Also $\psi^{'}_{u}(t^+) > 0$ and $\psi_{u}^{'}(t^-) < 0$ implies $t^+u \in \mathcal{N}^{+}_{\lambda}$ and $t^-u \in \mathcal{N}^{-}_{\lambda}.$  Since $\phi^{'}_{u}(t) = t^{q}(\psi_{u}(t)- \lambda \int_{\Omega} f(x)|u|^{q}).$ Then $\phi^{'}_{u}(t)<0$ for all $t \in [0, t^+)$ and $\phi^{'}_{u}(t)>0$ for all $t \in (t^+, t^-)$. So $\mathcal{J}_\lambda(t^+u) = \displaystyle\min_{0 \leq t \leq t^-}\mathcal{J}_\lambda (tu).$ Also $\phi^{'}_u(t) > 0$ for all $t \in [t^+, t^-),
\phi^{'}_u(t^-) = 0$ and $\phi^{'}_u(t) < 0$ for all $t \in (t^-, \infty)$ implies that $\mathcal{J}_\lambda(t^-u)
= \displaystyle\max_{t \geq t_{\max}} \mathcal{J}_\lambda(tu).$\\
(ii) Let $u\in H^+\cap G^-$. Then from equation \eqref{siutd}, we note that $\psi_{u}(t) \rightarrow \infty$ as $t \rightarrow \infty$. Also $\psi^{\prime}_{u}(t)>0$ for all $t>0$. Hence for all $\lambda>0$ there exists $t^*>0$ such that $t^*u\in \mathcal{N}_\lambda^+$ and $\mathcal{J}_\lambda(t^*u)=\displaystyle\min_{ t\geq 0} \mathcal{J}_\lambda(tu)$.

\end{proof}
\begin{lem}\label{L35}
There exists a constant $C_{2}>0$ such that
\begin{equation*}
    \theta_{\lambda}^+ \leq - \frac{(p-q)(r-p)}{p\;q\;r}a\;C_{2}<0
\end{equation*}
\end{lem}
\begin{proof}
Let $v_{\lambda} \in \mathrm{X}_{0}$ such that $\int_\Omega{f(x)|v_{\lambda}|^{q}dx}>0$. Then by Lemma \ref{L37}, there exists $ t_{\lambda}(v_{\lambda}) > 0$ such that  $t_{\lambda}v_{\lambda} \in \mathcal{N}^{+}$.
Since $t_{\lambda}v_{\lambda} \in \mathcal{N}_{\lambda}^+$, we have
\begin{equation*}
 \mathcal{J}_{\lambda}(t_{\lambda}v_{\lambda}) = \left(\frac{1}{p}-\frac{1}{q}\right)a \|t_{\lambda} v_{\lambda}\|_{X_0}^{p} + \left(\frac{1}{2p}-\frac{1}{q}\right)b \|t_{\lambda} v_{\lambda}\|_{X_0}^{2p} + \left(\frac{1}{q}-\frac{1}{r}\right) \int_\Omega g(x){|t_{\lambda}v_{\lambda}|^{r}dx}.
\end{equation*}
and
\begin{equation*}
  \int_\Omega g(x){|t_{\lambda}v_{\lambda}|^{r}dx} \leq \left(\frac{p-q}{r-q}\right)a\|t_\lambda v_\lambda\|_{X_0}^{p}+ \left(\frac{2p-q}{r-q}\right)b\|t_\lambda v_\lambda\|_{X_0}^{2p}.
\end{equation*}
Therefore
\[
\mathcal{J}_{\lambda}(t_{\lambda}v_{\lambda}) \leq -\frac{(p-q)(r-p)}{p\;q\;r} a\|t_{\lambda}v_{\lambda}\|_{X_0}^{p}
\leq  -\frac{(p-q)(r-p)}{p\;q\;r} a\; \mathrm{C_{2}},
\]
where $C_2=\|t_\lambda v_\lambda \|_{X_0}.$ This implies $ \theta_{\lambda}^+ \leq -\frac{(p-q)(r-p)}{p\;q\;r} a\; \mathrm{C_{2}}$.
\end{proof}
\begin{lem}\label{tt}
For a given ${u \in \mathcal{N}_{\lambda}}$ and $\lambda \in (0, \lambda_{0}),$ there exists $\epsilon > 0$ and a differentiable function
$\xi : \mathcal{B}(0,\epsilon) \subseteq X_{0} \rightarrow \mathbb{R}$ such that $\xi(0)=1,$ the function $\xi(v)(u-v)\in \mathcal{N}_{\lambda}$
and
\begin{equation}\label{3be}
\langle\xi^{'}(0), v\rangle = \frac{pa \langle u, v\rangle+ 2pb \|u\|_{X_0}^{p}\langle u, v\rangle - q\lambda \int_\Omega f(x)|u|^{q-2}u\;v\;dx- r\int g(x)|u|^{r-2}u\;v dx}{(p-q)a\|u\|_{X_0}^{p} + (2p-q)b\|u\|_{X_0}^{2p}-(r-q)\int_\Omega g(x)|u|^rdx}
\end{equation}
where $\langle u, v\rangle=\int_{\mathbb{R}^{2n}} \left|u(x)-u(y)\right|^{p-2}(u(x)-u(y))(v(x)-v(y))|x-y|^{-n-ps} dx\,dy\notag$
for all $v\in X_0$.
\end{lem}
\begin{proof}
For fixed ${u\in \mathcal{N}_\lambda}$, define $\mathcal{F}_u:\mathbb{R}\times X_{0} \rightarrow \mathbb{R}$ as follows
\begin{eqnarray*}
  \mathcal{F}_u(t,w) &=& t^{p}a\|u-w\|_{X_0}^{p}+t^{2p}b\|u-w\|_{X_0}^{2p}-t^{q}\lambda \int_\Omega{f(x)|u-w|^{q}dx}-t^{r}\int_\Omega g(x){|u-w|^{r}dx}
\end{eqnarray*}
then $\mathcal{F}_u(1,0) = 0,\; \frac{\partial}{ \partial t}\mathcal{F}_u(1,0)\neq 0$ as ${\mathcal{N}_{\lambda}^{0} = \phi}.$ So we can apply implicit function theorem to get a differentiable function $\xi : \mathcal{B}(0, \epsilon) \subseteq \mathrm{X}_{0} \rightarrow \mathbb{R}$ such that $\xi(0) = 1$
and equation \eqref{3be} holds and $\mathcal{F}_u(\xi(w),w) = 0$, $\textrm{for all}\; w \in \mathcal{B}(0, \epsilon)$.
This implies
\begin{equation*}
    a \|\xi(w)(u-w)\|_{X_0}^{p} + b \|\xi(w)(u-w)\|_{X_0}^{2p} - \lambda \int_\Omega{f(x)|\xi(w)(u-w)|^{q}dx} - \int_\Omega g(x){|\xi(w)(u-w)|^rdx} = 0.
\end{equation*}
Hence $ \xi(w)(u-w) \in \mathcal{N}_{\lambda}$.
\end{proof}
\begin{prop}\label{prp1}
Let $\lambda_0=\min\{\lambda_1, \lambda_2\}$, then for $\lambda \in (0,\lambda_{0})$ there exists a minimizing sequence $\{u_k\} \subset \mathcal{N}_{\lambda}$ such that
\begin{center}
    $\mathcal{J}_{\lambda}(u_{k}) = \theta_{\lambda}+o_k(1)$ and $\mathcal{J}_{\lambda}^{'}(u_{k}) = o_k(1).$
\end{center}
\end{prop}
\begin{proof}
Using Lemma \ref{le44} and  Ekeland variational principle \cite{eke1974} , there exists a minimizing sequence\\ $\{u_k\}\subset\mathcal{N}_\lambda $ such that
\begin{equation}\label{evp1}
\mathcal{J}_\lambda(u_k)< \theta_\lambda+\frac{1}{k}
\end{equation}
\begin{equation}\label{eevvpp}
\mathcal{J}_\lambda(u_k)< \mathcal{J}_\lambda(v)+\frac{1}{k}\|v-u_k\|_{X_0}\; \textrm{for each}\; v \in \mathcal{N}_\lambda
\end{equation}
Using equation \eqref{evp1}, Lemma \eqref{L35} and H\"older's inequality,  we get, $u_k\not \equiv 0$.
Next, we claim that $\|\mathcal{J}^{'}_\lambda(u_k)\|\rightarrow 0$ as $k\rightarrow \infty$.
Now, using the Lemma \ref{tt} we get the differentiable functions $\xi_k:\mathcal{B}(0, \epsilon_k)\rightarrow \mathbb{R}$ for some $\epsilon_k>0$ such that $\xi_k(v)(u_k-v)\in \mathcal{N}_\lambda$,\; $\textrm{for all}\;\; v\in \mathcal{B}(0, \epsilon_k).$
\noindent For fixed $k$, choose $0<\rho<\epsilon_k$. Let $u\in X_0$ with $u\not\equiv 0$ and let $v_\rho=\frac{\rho u}{\|u\|_{X_0}}$. We set $\eta_\rho=\xi_k(v_\rho)(u_k-v_\rho)$. Since $\eta_\rho \in \mathcal{N}_\lambda$, we get from equation \eqref{eevvpp}
\begin{align*}
\mathcal{J}_\lambda(\eta_\rho)-\mathcal{J}_\lambda(u_k)\geq-\frac{1}{k}\|\eta_\rho-u_k\|_{X_0}
\end{align*}
{Using mean value theorem, $\displaystyle\lim_{n\rightarrow \infty}\frac{|\xi_k(v_\rho)-1|}{\rho}\leq \|\xi_k(0)\|_{X_0}$ and taking limit  $\rho\rightarrow 0$ , we get}
\begin{equation}
\langle \mathcal{J}^{\prime}_\lambda(u_k),\frac{u_k}{\|u_k\|_{X_0}}\rangle\leq\frac{C}{k}(1+\|\xi_k^{'}(0)\|_{X_0})
\end{equation}
for some constant $C>0$, independent of $\rho$. So if we can show that $\|\xi_k^{'}(0)\|_{X_0}$ is bounded then we are done. Now from Lemma \ref{tt} and H\"older's inequality, we get
\begin{equation}
\langle \xi_k^{'}(0), v\rangle=\frac{K\|v\|_{X_0}}{(p-q)a\|u_k\|_{X_0}^{p} + (2p-q)b\|u_k\|_{X_0}^{2p}-(r-q)\int_\Omega g(x)|u_k|^rdx} \;\textrm{for some}\; K>0
\end{equation}
So to prove the claim we only need to prove that
\; $(p-q)a\|u_k\|_{X_0}^{p} + (2p-q)b\|u_k\|_{X_0}^{2p}-(r-q)\int_\Omega g(x) |u_k|^rdx$ is bounded away from zero. Suppose not. Then there exists a subsequence $u_k$ such that
\begin{equation}\label{baw}
(p-q)a\|u_k\|_{X_0}^{p} + (2p-q)b\|u_k\|_{X_0}^{2p}-(r-q)\int_\Omega g(x)|u_k|^rdx=o_k(1)
\end{equation}
From equation \eqref{baw}, we get $E_\lambda(u_k)=o_k(1)$ and
\begin{equation}\label{inequ}
\|u_k\|_{X_0} \geq \left[\left(\frac{p-q}{r-q}\right)\frac{aS_r^{r}}{\|g\|_\infty}\right]^{\frac{1}{r-p}}+o_k(1).
\end{equation}
Now following the proof of Lemma \ref{2s1}, we get $E_\lambda(u_k)>0$ for large $k$, which is a contradiction. Hence $\{u_k\}$ is a Palais-Smale sequence for $\mathcal{J}_\lambda.$

\end{proof}
\noindent \textbf{Proof of Theorem \ref{th1} :} Using proposition \ref{prp1} and Lemma \ref{pscn}, there exist minimizing subsequences $\{u_k^\pm\}\in \mathcal{N}_\lambda^\pm $ and $u_\lambda^\pm\in X_0$ such that $u_k^\pm\rightarrow u_\lambda^\pm$ strongly in $X_0$ for $\lambda\in(0, \lambda_0)$. Therefore for $\lambda\in(0, \lambda_0)$,\;$u_\lambda^\pm$ are weak solutions of problem $(P_\lambda)$. Hence $u_\lambda^\pm\in \mathcal{N}_\lambda$. Moreover $u_\lambda^\pm\in \mathcal{N}_\lambda^\pm$ and $\mathcal{J}_\lambda(u_\lambda^\pm)=\theta_\lambda^\pm$. As $\mathcal{J}_\lambda(u_\lambda^\pm)=\mathcal{J}_\lambda(|u_\lambda^\pm|)$ and $|u_\lambda^\pm|\in \mathcal{N}_\lambda^\pm$, by Lemma \ref{lmcp}, $u_\lambda^\pm$ are non-negative solutions of {$(P_\lambda)$}. Also,\;$\mathcal{N}_\lambda^+\cap \mathcal{N}_\lambda^-=\emptyset$, implies that $u_\lambda^+$ and $u_\lambda^-$ are distinct solutions.\\
Now to prove Theorem \ref{th22}, we need following Lemmas.
\begin{lem} \label{repd}
Let $r=2p$ and  $\Lambda$ be defined as in \eqref{minmm}. Then
\begin{enumerate}[(i)]
\item if $b\ge\frac{1}{\Lambda}$, then $\mathcal{N}_\lambda^+=\mathcal{N}_\lambda$ \;\textrm{for all}\; $\lambda>0$.
\item if $b<\frac{1}{\Lambda}$, then there exists a $\lambda^0>0$ such that ${\mathcal{N}_\lambda^0=\phi}\;, \textrm{for all}\; \lambda\in(0, \lambda^0)$.
\end{enumerate}
\end{lem}
\begin{proof}
(i) From equation \eqref{phid2}, for any $u\in \mathcal{N}_\lambda,$ we have
\begin{align*}
\phi_u^{\prime\prime}(1)&=a(p-q)\|u\|_{X_0}^p+(2p-q)\left(b\|u\|_{X_0}^{2p}-\int_\Omega g(x) |u|^{2p}dx\right)\\
&>a(p-q)\|u\|_{X_0}^p+\frac{(2p-q)(b\Lambda-1)}{\Lambda}\|u\|_{X_0}^{2p}>0.
\end{align*}
So $u\in \mathcal{N}_\lambda^+$, implies $\mathcal{N}_\lambda^+=\mathcal{N}_\lambda$.\\
(ii) Suppose $u\in \mathcal{N}_\lambda^0$. Then by equation \eqref{phid2}
\begin{align}
(2p-q)\left(\int_\Omega g(x)|u|^{2p}dx-b\|u\|_{X_0}^{2p}\right)&=a(p-q)\|u\|_{X_0}^p\\\label{t2}
\|u\|_{X_0}&\geq\left(\frac{a(p-q)\Lambda}{(2p-q)(1-b\Lambda)}\right)^\frac{1}{p}.
\end{align}
Also, from equations \eqref{phid1}, \eqref{phid2}, H\"older's inequality and  Sobolev inequality
we get ,
\begin{equation}\label{t1}
\|u\|_{X_0}\leq\left(\frac{\lambda(2p-q)l^{\frac{r}{r-q}}}{a p S_r^{q}}\right)^{\frac{1}{p-q}},
\end{equation}
where $S_r$ is the Sobolev constant for embedding of $X_0$ in $L^r(\Omega)$. Now, from equations \eqref{t2} and \eqref{t1}, we get
\begin{equation}\label{cbnd}
\left(\frac{a(p-q)\Lambda}{(2p-q)(1-b\Lambda)}\right)^\frac{1}{p}\le \|u\|_{X_0}\leq\left(\frac{\lambda(2p-q)l^{\frac{r}{r-q}}}{a p S_r^{q}}\right)^{\frac{1}{p-q}}.
\end{equation}
So equation \eqref{cbnd} holds if  $\lambda\geq\lambda^0=\frac{paS_r^{q}}{(2p-q)l^{\frac{r}{r-q}}}\left(\frac{a\Lambda(p-q)}{(1-b\Lambda)(2p-q)}\right)^{\frac{p-q}{p}}$. Thus, for $\lambda<\lambda^0, \;{\mathcal{N}_\lambda^0=\phi} $
\end{proof}
\begin{lem}\label{repes}
For $r=2p$ and $b\geq\frac{1}{\Lambda}$, there exists a unique $0<t^{+}<t_{\max}$ such that $t^+u \in \mathcal{N}_\lambda$, whenever $u\in H^+$. Also
\begin{align*}
\mathcal{J}_\lambda(t^+u)=\displaystyle\inf_{t\geq0}\mathcal{J}_\lambda(tu).
\end{align*}
\end{lem}
\begin{proof}
Using equation {\eqref{siut} and \eqref{siutd}} for $r=2p$, we have
\begin{eqnarray}\label{siutrepp}
\psi_u^(t)&=& {at^{p-q}\|u\|_{X_0}^p+t^{2p-q}\left(b\|u\|_{X_0}^{2p}-\int_\Omega g(x)|u|^{2p}dx \right)},\\\label{siutrep}
 \psi_u^{\prime}(t)&=&a(p-q)t^{p-q-1}\|u\|_{X_0}^p+(2p-q)t^{2p-q-1}\left(b\|u\|_{X_0}^{2p}-\int_\Omega g(x)|u|^{2p}dx \right).
\end{eqnarray}
Observe that $\psi_u^{\prime}(t)>0$ for $t>0$. So for $u\in H^+$ there exists a unique $t^+(u)>0$ such that $t^+u\in \mathcal{N}_\lambda=\mathcal{N}_\lambda^+.$ Now using equation \eqref{psr}, we get
$\frac{d}{dt}\mathcal{J}_\lambda(t^+u)=0,\; \frac{d^2}{dt^2}\mathcal{J}_\lambda(t^+u)>0$ for all $t>0$. Hence $\mathcal{J}_\lambda(t^+u)=\displaystyle\inf_{t\geq0}\mathcal{J}_\lambda(tu)$.
\end{proof}
\begin{lem}\label{repex}
For $r=2p$ and $b\leq\frac{1}{\Lambda}$, there exists a unique $t_{\max}(u)>0$ such that
\begin{enumerate}[(i)]
\item For $u\in H^+$, there exists $\lambda^0>0$ and unique $t^+(u)<t_{\max}(u)<t^-(u)$ such that $t^{\pm} u\in \mathcal{N}_\lambda^{\pm}$ and $\mathcal{J}_\lambda(t^+u)=\displaystyle\inf_{0\leq t\leq t_{\max}}\mathcal{J}_\lambda(tu), \mathcal{J}_\lambda(t^-u)=\displaystyle\sup_{t\geq t_{\max}}\mathcal{J}_\lambda(tu)$ for all $\lambda\in(0, \lambda^0)$.
\item For $u\in H^-$, there exists unique $t^-(u)>t_{\max}$ such that $t^-u\in \mathcal{N}_\lambda^-$ and $\mathcal{J}_\lambda(t^-u)=\displaystyle\sup_{t\geq 0}\mathcal{J}_\lambda(tu)$
     \end{enumerate}
\end{lem}
\begin{proof}
(i) From the equation \eqref{siutrep} {and \eqref{siutrepp}}, we have $t_{\max}>0$ such that
\begin{align*}
\psi_u(t_{\max})\geq\frac{pa^{\frac{2p-q}{p}}}{2p-q}\left(\frac{(p-q)\Lambda}{(2p-q)(1-\Lambda b)}\right)^{\frac{p-q}{p}}\|u\|_{X_0}^q
\end{align*}
Now if $\lambda<\lambda^0,$ we have unique $t^+<t_{\max}<t^-$ such that $\psi_u(t^\pm)=\lambda\int_\Omega f(x)|u|^qdx$  that is,  $t^\pm u\in \mathcal{N}_\lambda$. Since $\psi^{\prime}_u(t)>0$ for $0<t<t_{\max}$ and $\psi^{\prime}_u(t)<0$ for $t>t_{\max}$ we get $\psi^{\prime}_u(t^+)>0$ and $\psi^{\prime}_u(t^-)<0$. So by the equation \eqref{psr}, $t^\pm u\in \mathcal{N}_\lambda^\pm$.\\
\noindent Again from equation \eqref{psr}, we observe
$\frac{d}{dt}\mathcal{J}_\lambda(t^+u)=0,\; \frac{d^2}{dt^2}\mathcal{J}_\lambda(t^+u)>0$ for all $0<t<t_{\max}$. Hence $\mathcal{J}_\lambda(t^+u)=\displaystyle\inf_{0<t<t_{\max}}\mathcal{J}_\lambda(tu)$. Similarly,
$\frac{d}{dt}\mathcal{J}_\lambda(t^-u)=0,\; \frac{d^2}{dt^2}\mathcal{J}_\lambda(t^-u)<0$ for all $t>t_{\max}$. Hence $\mathcal{J}_\lambda(t^-u)=\displaystyle\sup_{t>t_{\max}}\mathcal{J}_\lambda(tu)$.\\
(ii) For $u\in H^-$, $\psi^{\prime}_u(t)<0$ for $t>t_{\max}$ and $\psi_u(t)\rightarrow -\infty$ as $t\rightarrow \infty$. So we have a unique $t^-(u)>t_{\max}$ such that $\psi_u(t^-)=\lambda\int_\Omega f(x)|u|^qdx$. From equation \eqref{psr}, $t^-u\in \mathcal{N}_\lambda^-$ and $\mathcal{J}_\lambda(t^-u)=\displaystyle\sup_{t\geq 0}\mathcal{J}_\lambda(tu)$.
\end{proof}
\begin{lem}\label{repem}
Assume $r=2p,\;b<\frac{1}{\Lambda}$. Then  $\theta_\lambda^+<0$.
\end{lem}
\begin{proof}
Let $w_\lambda\in X_0$ be such that $\int_\Omega f(x)|w_\lambda|^qdx >0$. Then from Lemma \ref{repex}, $\mathcal{N}_\lambda^+\neq \emptyset$. Take $u\in \mathcal{N}_\lambda^+$, we have
\begin{equation}\label{so11}
a\|u\|_{X_0}^p+b\|u\|_{X_0}^{2p}-\lambda\int_\Omega f(x)|u|^qdx-\int_\Omega g(x)|u|^{2p}dx=0
\end{equation}
\begin{equation}\label{so22}
pa\|u\|_{X_0}^p+2pb\|u\|_{X_0}^{2p}-\lambda q\int_\Omega f(x)|u|^qdx-2p\int_\Omega g(x)|u|^{2p}dx>0
\end{equation}
Now multiplying equation\eqref{so11} by $2p$ and subtracting from equation \eqref{so22}, we get
\begin{align}\label{re1}
\lambda(2p-q)\int_{\Omega} f(x)|u|^qdx\ge pa\|u\|_{X_0}^p
\end{align}
Now using equation \eqref{so11} and \eqref{re1} , we get
\begin{align*}
\mathcal{J}_\lambda (u)&=\frac{1}{p}\widehat{M}(\|u\|_{X_0}^p)-\frac{\lambda}{q}\int_\Omega f(x)|u|^qdx-\frac{1}{2p}\int_\Omega g(x)|u|^{2p}dx,\\
&=a\left(\frac{1}{p}-\frac{1}{2p}\right)\|u\|_{X_0}^p-\lambda\left(\frac{1}{q}-\frac{1}{2p}\right)\int_\Omega f(x)|u|^q dx\\
&=\frac{a}{2p}\|u\|_{X_0}^p-\lambda\frac{(2p-q)}{2pq}\int_\Omega f(x)|u|^q dx\\
&={\frac{1}{2p}a\|u\|_{X_0}^p-\frac{1}{2q}a\|u\|_{X_0}^p<-\frac{(p-q)}{2pq}a\|u\|_{X_0}^{p}<0}.
\end{align*}
Hence $\theta_\lambda^+<0$.
\end{proof}
\noindent \textbf{Proof of Theorem \ref{th22}:} Using Lemmas \ref{repd}, \ref{repes}, \ref{repex}, \ref{repem}, \ref{pscn} and repeating the argument given in proof of Theorem \ref{th1}, we get two non-negative solutions $u_\lambda^+\in \mathcal{N}_\lambda^+$ and $u_\lambda^-\in \mathcal{N}_\lambda^-$.\\
\noindent \textbf{Proof of Theorem \ref{th33}:}\\
Now to prove Theorem \ref{th33}, we use truncation technique as used in \cite{col} and \cite{cyuyc} for Laplacian operator.
{We consider the following truncated problem,}
\begin{equation}\label{fraclapl}
(P_{\lambda, k})\;
\left\{\begin{array}{rllll}
M_k\left(\displaystyle\int_{\mathbb{R}^{2n}}\frac{|u(x)-u(y)|^p}{\left|x-y\right|^{n+ps}}dx\,dy\right)(-\Delta)^{s}_p u
&=\lambda f(x)|u|^{q-2}u+ g(x)\left|u\right|^{r-2}u\;  \text{in } \Omega,\\
u&=0 \;\mbox{in } \mathbb{R}^{n}\setminus \Omega,
\end{array}
\right.
\end{equation}
where $k\in \left(\frac{a(r-p)}{rb}, \frac{a(r-p)}{pb}\right)$ and
\begin{equation}\label{trun}
M_k(t)=
\left\{\begin{array}{ll}
M(t) & \text{if } 0\leq t\leq k,\\
M(k) & \text{if } t>k,
\end{array}\right.
\end{equation}
is a truncation of $M(t)$.\\
The energy functional associated to the problem is given as
\begin{equation}\label{enjk}
\mathcal{J}_{\lambda, k} (u)=\frac{1}{p}\widehat{M_k}(\|u\|_{X_0}^p)-\frac{\lambda}{q}\int_\Omega f(x)|u|^qdx-\frac{1}{r}\int_\Omega g(x)|u|^rdx
\end{equation}
where {$\widehat{M_k}$ is the primitive of $M_k$. We define}
\begin{align*}
\mathcal{N}_{\lambda, k}=\{{u\in X_0\setminus \{0\}} :\langle \mathcal{J}_{\lambda, k}^\prime(u), u\rangle=0\}
\end{align*}
Now fiber maps $\phi_{u,k}: \mathbb R^+\rightarrow \mathbb R$ are
defined as $\phi_{u,k}(t)=\mathcal{J}_{\lambda, k}(tu)$. For
$u\in X_0$, we have
\begin{align}
\phi_{u,k}(t) &=\frac{1}{p}\widehat{M_k}(t^p\|u\|_{X_0}^p)-\frac{\lambda}{q}t^q\int_\Omega f(x)|u|^qdx-\frac{1}{r}t^r\int_\Omega g(x)|u|^rdx,\\\label{mk1}
\phi_{u,k}^{\prime}(1) &=M_k(\|u\|_{X_0}^p)\|u\|_{X_0}^p-\lambda \int_\Omega f(x)|u|^qdx-\int_\Omega g(x)|u|^rdx,\\\label{mk2}
\phi_{u,k}^{\prime\prime}(1) &= (p-1)M_k(\|u\|_{X_0}^p)\|u\|_{X_0}^p+M_k^{\prime}(\|u\|_{X_0}^p)p\|u\|_{X_0}^{2p}\\
&-\lambda(q-1) \int_\Omega f(x)|u|^qdx-(r-1)\int_\Omega g(x)|u|^rdx.
\end{align}
Then it is easy to see that $u\in \mathcal N_{\lambda,k}$ if and only if
$\phi_{u,k}^{\prime}(1)=0$. That is,
\begin{equation}\label{Meq2}
M_k(\|u\|_{X_0}^p)\|u\|_{X_0}^p - \lambda \int_{\Omega} f(x)|u|^q dx-
\int_{\Omega}g(x)|u|^{r}dx =0 .
\end{equation}
We split
$\mathcal {N}_{\lambda,k}$ into three parts corresponding to local minima,
local maxima and points of inflection. For this, we set
\begin{align*}
\mathcal N_{\lambda,k}^{\pm}&:= \left\{u\in \mathcal N_{\lambda,k}:
\phi_{u,k}^{\prime\prime}(1)
\gtrless0\right\} =\left\{tu\in X_0 : \phi_{u,k}^{\prime}(t)=0,\; \phi_{u,k}^{''}(t)\gtrless  0\right\},\\
\mathcal N_{\lambda,k}^{0}&:= \left\{u\in \mathcal N_{\lambda,k}:
\phi_{u,k}^{\prime\prime}(1) = 0\right\}=\left\{tu\in X_{0} :
\phi_{u,k}^{\prime}(t)=0,\; \phi_{u,k}^{''}(t)= 0\right\}.
\end{align*}
We need following results to prove theorem \ref{th33}.
\begin{lem}
The energy functional $ \mathcal{J}_{\lambda, k}$ is coercive and bounded below on $ \mathcal{N}_{\lambda, k}$.
\end{lem}
\begin{proof}
Using equation \eqref{Meq2} for $u\in \mathcal{N}_{\lambda, k}$, we get
\begin{align*}
\mathcal{J}_{\lambda, k} (u)&=\frac{1}{p}\widehat{M_k}(\|u\|_{X_0}^p)-\frac{\lambda}{q}\int_\Omega f(x)|u|^qdx-\frac{1}{r}\int_\Omega g(x)|u|^rdx,\\
&=\frac{1}{p}\widehat{M_k}(\|u\|_{X_0}^p)-\frac{1}{r}{M_k}(\|u\|_{X_0}^p)\|u\|_{X_0}^p-\lambda\left(\frac{1}{q}-\frac{1}{r}\right)\int_\Omega f(x)|u|^qdx,\\
&>\left(\frac{a}{p}-\frac{M(k)}{r}\right)\|u\|_{X_0}^p-\lambda\left(\frac{1}{q}-\frac{1}{r}\right)l^{\frac{r}{r-q}}S_r^{-q}\|u\|_{X_0}^q.
\end{align*}
Thus $ \mathcal{J}_{\lambda, k}$ is coercive and bounded below as {$\frac{a}{p}>\frac{M(k)}{r}$} for selected choice of $k$.
\end{proof}
\begin{lem}
There exists a $\lambda^1>0$ such that $ \mathcal{N}_{\lambda, k}^0=\emptyset$ for all $\lambda\in (0,\lambda^1)$.
\end{lem}
\begin{proof}
Suppose not, then for  $u\in  \mathcal{N}_{\lambda, k}^0$ we have two cases\\
\textbf{Case 1:}  $\|u\|_{X_0}^p\leq k$\\
 From equations \eqref{mk1} and \eqref{mk2}, we have
\begin{align}\label{tt1}
(p-q)a\|u\|_{X_0}^p+(2p-q)b\|u\|_{X_0}^{2p}-(r-q)\int_\Omega g(x)|u|^r dx =0\\ \label{tt2}
(r-p)a\|u\|_{X_0}^p-(2p-r)b\|u\|_{X_0}^{2p}-\lambda(r-q)\int_\Omega f(x)|u|^q dx =0
\end{align}
\textbf{Case 2:}  $\|u\|_{X_0}^p>k$
\begin{align}\label{tt3}
(p-q)M(k)\|u\|_{X_0}^p-(r-q)\int_\Omega g(x) |u|^r dx =0. \\\label{tt4}
(r-p)M(k)\|u\|_{X_0}^p-\lambda(r-q)\int_\Omega f(x)|u|^q dx =0.
\end{align}
Now from equations {\eqref{tt1}, \eqref{tt3}} and H\"older's inequality, we get
\begin{align} \label{ttt1}
(p-q)a\|u\|_{X_0}^p\leq (r-q)\|g\|_\infty S_r^{-r}\|u\|_{X_0}^r. \\\label{ttt2}
(p-q)M(k)\|u\|_{X_0}^p\leq (r-q)\|g\|_\infty S_r^{-r}\|u\|_{X_0}^r.
\end{align}
From {\eqref{ttt1} and \eqref{ttt2}}, we get
\begin{align}\label{tT1}
C_1\|u\|_{X_0}^p \leq (r-q)  \|g\|_\infty S_r^{-r}\|u\|_{X_0}^r .
\end{align}
{where $C_1=(p-q)\min \{a, M(k)\}$. Similarly, from \eqref{tt2} and \eqref{tt4} and H\"older's inequality, we get}
\begin{align}\label{ttt3}
[(r-p)a-(2p-r)bk]\|u\|_{X_0}^{p}\leq \lambda(r-q)l^{\frac{r}{r-q}}S_{r}^{-q}\|u\|_{X_0}^q. \\\label{ttt4}
(r-p)M(k)\|u\|_{X_0}^p\leq\lambda(r-q)l^{\frac{r}{r-q}}S_{r}^{-q}\|u\|_{X_0}^q.
\end{align}
{From \eqref{ttt3} and \eqref{ttt4},} we get
\begin{align}\label{tT2}
C_2\|u\|_{X_0}^p \leq \lambda(r-q)l^{\frac{r}{r-q}}S_{r}^{-q}\|u\|_{X_0}^q.
\end{align}
{where $C_2=\min \{[(r-p)a-(2p-r)bk], (r-p)M(k)\}$. From \eqref{tT1} and \eqref{tT2},} we get
\begin{align}\label{fe}
\left(\frac{C_1S_r^{r}}  {(r-q)\|g\|_\infty}\right)^{\frac{1}{r-p}}\leq \|u\|_{X_0}\leq \left(\frac  {\lambda(r-q)l^{\frac{r}{r-q}}}{C_2S_{r}^{q}}\right)^{\frac{1}{p-q}}.
\end{align}
{Hence \eqref{fe} holds if} $\lambda \geq \lambda^1=\left(\frac{C_1S_r^{r}}  {(r-q)\|g\|_\infty}\right)^{\frac{p-q}{r-p}}\left(\frac  {C_2S_{r}^{q}}{(r-q)l^{\frac{r}{r-q}}}\right)$. Hence for $\lambda <\lambda^1$,\; $\mathcal{N}_{\lambda, k}^0=\phi$. \end{proof}
\noindent Next, we define $\theta_{\lambda, k}^+=\displaystyle\inf_{u\in\mathcal{N}_{\lambda, k}} \mathcal{J}_{\lambda, k}(u)$. Then we have following Lemma.
\begin{lem}\label{lem415}
$\theta_{\lambda, k}^+<0,\; \text{for all} \;\lambda \in (0, \lambda^1).$
\end{lem}
\begin{proof}
Let $u\in \mathcal{N}_{\lambda, k}^+$ Then we have two cases.\\
\textbf{Case 1:} $\|u\|_{X_0}^p\leq k$\\
From \eqref{ttt3}, we have
\begin{align}\label{ttt3e}
[(r-p)a-(2p-r)bk]\|u\|_{X_0}^{p}\leq \lambda(r-q)\int_\Omega f(x)|u|^qdx.
\end{align}
Now, using  \eqref{ttt3e}, we have
\begin{align*}
\mathcal{J}_{\lambda, k} (u)&=\frac{1}{p}\widehat{M_k}(\|u\|_{X_0}^p)-\frac{1}{r}{M_k}(\|u\|_{X_0}^p)\|u\|_{X_0}^p-\lambda\left(\frac{1}{q}-\frac{1}{r}\right)\int_\Omega f(x)|u|^qdx\\
&=\left(\frac{1}{p}-\frac{1}{r}\right)a\|u\|_{X_0}^p+\left(\frac{1}{2p}-\frac{1}{r}\right)b\|u\|_{X_0}^{2p}-\lambda\left(\frac{1}{q}-\frac{1}{r}\right)\int_\Omega f(x)|u|^qdx\\
&\leq -\frac{(p-q)(r-p)}{pqr}a\|u\|_{X_0}^p +\frac{(2p-q)(2p-r)}{2pqr}b\|u\|_{X_0}^{2p}\\
&\leq -\frac{k}{pqr}\left((r-p)a-(2p-r)bk\right)<0.
\end{align*}
\textbf{Case 2:} $\|u\|_{X_0}^p> k$\\
From  \eqref{tt4}, we have
\begin{align}\label{tt4e}
0<(r-p)M(k)\|u\|_{X_0}^p<\lambda(r-q)\int_\Omega f(x)|u|^q dx.
\end{align}
Also
\begin{align*}
\widehat{M_k}(t)=\int_0^t M_k(s)ds=\int_0^k M(s)ds+\int_k^t M_k(s)ds= \widehat{M}(k)+M(k)(t-k)\; \textrm{for}\; t>k.
\end{align*}
Using  \eqref{tt4e}, we get
\begin{align*}
\mathcal{J}_{\lambda, k} (u)&=\frac{1}{p}\widehat{M_k}(\|u\|_{X_0}^p)-\frac{1}{r}{M_k}(\|u\|_{X_0}^p)\|u\|_{X_0}^p-\lambda\left(\frac{1}{q}-\frac{1}{r}\right)\int_\Omega f(x)|u|^qdx\\
&=\frac{1}{p}[\widehat{M}(k)-M(k)k]+(\frac{1}{p}-\frac{1}{r})M(k)\|u\|_{X_0}^p-\lambda\left(\frac{1}{q}-\frac{1}{r}\right)\int_\Omega f(x)|u|^qdx\\
&<-\frac{1}{2p}bk^2-\lambda \frac{(r-q)(p-q)}{pqr}\int_\Omega f(x)|u|^q dx <0,
\end{align*}
Hence $\theta_{\lambda, k}^+<0$.
\end{proof}
\noindent
\textbf{Proof of Theorem \ref{th33}(i):} Repeating the same argument as in Theorem \ref{th1} and using Lemma \ref{pscn} and the fact that $\mathcal{J}_{\lambda, M}$ is bounded below in $X_0$, we get a Palais-Smale sequence which converges strongly.\\
\noindent \textbf{Proof of Theorem \ref{th33}(ii):} Let $\theta>0$ and take $\lambda<\hat{\lambda}_0=\min\{\theta, \lambda^1\}$. Then by Lemma \ref{lem415} and  proposition \ref{prp1}, there exists two minimizing sequences satisfying \eqref{pscn1}. Using Lemma \ref{pscn}, up to subsequences, $u_k^\pm\rightarrow u_\lambda^\pm$ strongly in $X_0$. Hence $u_\lambda^\pm\in \mathcal{N}_{\lambda, k}^\pm$ and $\mathcal{J}_{\lambda, k}(u_\lambda^\pm)=\theta_\lambda^\pm$. As $\mathcal{J}_{\lambda, k}(u_\lambda^\pm)=\mathcal{J}_{\lambda, k}(|u_\lambda^\pm|)$ and $|u_\lambda^\pm|\in \mathcal{N}_{\lambda, k}^\pm$, by Lemma \ref{lmcp}, $u_\lambda^\pm$ are non-negative solutions of $P_{\lambda, k}$. Also,\;$\mathcal{N}_{\lambda, k}^+\cap \mathcal{N}_{\lambda, k}^-=\phi$, implies that $u_\lambda^+$ and $u_\lambda^-$ are distinct solutions of problem $P_{\lambda, k}$. Now we claim that $\|u_{\lambda, k}^\pm\|^p\leq k$. If not then by Lemma \ref{two},
\begin{align*}
\frac{a(r-p)}{brL(\theta)}&<\frac{k}{L(\theta)}
<\frac{\|u_{\lambda, k}^\pm\|_{X_0}^p}{L(\lambda)}\\
&<\max\{M(k)^{\frac{q-r+2}{r-1}}, M(k)^{\frac{2}{r-1}}\}=\hat{A}
 \end{align*}
 which implies $b>\frac{a(r-p)}{p\hat{A}L(\theta)}$, a contradiction. Hence $u_\lambda^\pm\in \mathcal{N}_\lambda^\pm$ and  $u_\lambda^+$ and $u_\lambda^-$ are distinct solutions of problem $P_{\lambda}$.
\section{Existence of solution for critical case}
\setcounter{equation}{0}
\noindent To show the existence result in the case of critical nonlinearity, we need the following Lemma.
\begin{lem}\label{cmcr}
Suppose $\{u_j\}$ be a sequence in $X_0$ such that $\mathcal{J}_\lambda(u_j)\rightarrow c<\left(\frac{p_s^*-2p}{2pp_s^*}\right)\displaystyle\frac{(m_0 C)^{n/ps}}{S^{(n-ps)/ps}}-C\lambda^{\frac{p}{p-q}}$ and $\mathcal{J}'_\lambda(u_j)\rightarrow 0$, where C is positive constant depending on $p$ and $q$, then there exists a strongly convergent subsequence.
\end{lem}
\begin{proof}
By  equation \eqref{pscn1},  $\{u_j\}$ is bounded in $X_0$ and so upto subsequence $\{u_j\}$ converges weakly to $u$ in $X_0$, strongly  in $L^q$ for all $1\leq q< p_s^*$ and point wise to $u$ almost everywhere  in $\Omega$. Also there exists $h\in L^p(\Omega)$ such that $|u_j(x)|\leq h(x) $ a.e. in $\Omega$. Also, $\{\|u_j\|_{X_0}\}$ as a real sequence converges to $\alpha$ (say). Since $M$ is continuous, $M(\|u_j\|^p_{X_0})\rightarrow M(\alpha^p)$. Now we claim that
\begin{equation} \label{claim}
\left\|u_j\right\|^{p}_{X_0}\rightarrow \left\|u\right\|^{p}_{X_0}\qquad\text{as}\,\,j\rightarrow+\infty,
\end{equation}
Once claim is proved, we can invoke Brezis-Leib Lemma \cite{bre1983} to prove that $u_j$ converges to $u$ strongly in $X_0$.
We know that $\{u_j\}$ is also bounded in $W_0^{s,p}(\Omega)$. So we may assume that there exists two positive measures $\mu$ and $\nu$ on $\mathbb R^n$ such that
\begin{equation}\label{prkh1}
\left|(-\Delta)^s_p u_j\right|^p dx\stackrel{*}{\rightharpoonup}\mu\quad\mbox{and}\quad\left|u_j\right|^{p_s^*}\rightharpoonup\nu,
\end{equation}
in the sense of measure. Moreover, from (\cite{pal}), we have a countable index set $J$, positive constants $\{\nu_j\}_{j\in J}$ and $\{\mu_j\}_{j\in J}$ such that
\begin{equation}
\nu=\left|u\right|^{p_s^*}dx+\sum_{i\in J} \nu_i\delta_{x_i}, \;\; \; \text{and}
\end{equation}
\begin{equation}\label{prkh2}
\mu\geq\left|(-\Delta)^s_p u\right|^p dx+\sum_{i\in J} \mu_i\delta_{x_i},\qquad\nu_i\leq S\mu^{p_s^*/p}_{i},
\end{equation}
where $S$ is the best constant of the embedding $W^{s,p}_{0}(\Omega)$ into $L^{p_s^*}(\Omega)$.
Our goal is to show that $J$ is empty. Suppose not, then there exists $i\in J$.  For this $x_i$, define $\phi_\delta^i(x)=\phi(\frac{x-x_i}{\delta}),\ x\in \mathbb R^n$ and $\phi \in C_0^\infty(\mathbb R^n,[0,1])$ such that $\phi=1 \;\text{in}\; B(0,1)$ and $\phi=0 \;\text{in}\; \mathbb R^n\setminus B(0,2)$. Since $\{\phi_\delta^iu_j\}$ is bounded in $X_0$, we have $\mathcal J'_{a,\,\lambda}(u_j)(\phi_\delta^i u_j)\rightarrow 0$ as $j\rightarrow+\infty.$ That is,
\begin{align}\label{3.27}
M&(\left\|u_j\right\|^{p}_{X_0})\int_{\mathbb{R}^{2n}} \nonumber u_j(x)|(u_j(x)-u_j(y)|^{p-2}(u_j(x)-u_j(y))\big(\phi_\delta^i(x)-\phi_\delta^i(y)\big)|x-y|^{-n-ps}\,dx\,dy\nonumber\\
&=-M(\|u_j\|^{p}_{X_0})\int_{\mathbb{R}^{2n}} \phi_\delta^i(y)|u_j(x)-u_j(y)|^p|x-y|^{-n-ps}\,dx\,dy\notag\\
&\quad+\lambda\int_\Omega f(x)|u_j(x)|^{q-2}u_j(x)\phi_\delta^i(x)dx+\int_\Omega \left|u_j(x)\right|^{p_s^*}\phi_\delta^i(x)dx\,+o_j(1),
\end{align}
as $j\rightarrow\infty$.
Now using H\"older's inequality and the fact that $\{u_j\}$ is  bounded in $X_0$,  we get
\begin{align}
\big{|}\int_{\mathbb{R}^{2n}}u_j(x)&|u_j(x)-u_j(y)|^{p-2}(u_j(x)-u_j(y))\big(\phi_\delta^i(x)-\phi_\delta^i(y)\big)|x-y|^{-n-ps}\,dx\,dy \big{|} \notag\\
&\leq C \left(\int_{\mathbb{R}^{2n}} \left|u_j(x)\right|^p\left|\phi_\delta^i(x)-\phi_\delta^i(y)\right|^p|x-y|^{-n-ps}\,dx\,dy\right)^\frac{1}{p}.
\end{align}
Now we claim that
\begin{equation}
\lim_{\delta\rightarrow 0}\left[\lim_{j\rightarrow+\infty}\left(\int_{\mathbb{R}^{2n}}
\left|u_j(x)\right|^p\left|\phi_\delta^i(x)-\phi_\delta^i(y)\right|^p|x-y|^{-n-ps}\,dx\,dy\right)\right]=0.
\end{equation}
Using the Lipschitz regularity of $\phi_\delta^i$, we have, for some $L\ge 0,$
\begin{align*}
\int_{\mathbb{R}^{2n}} |u_j(x)|^p |\phi_{\delta}^{i}(x)\nonumber &-\phi_{\delta}^{i}(y)|^p |x-y|^{-n-ps}\,dx\,dy
\end{align*}
\begin{align}\label{facl}
&\leq\frac{1}{\theta}\int_{\mathbb{R}^{2n}} \left|u_j(x)\right|^p\left|\phi_\delta^i(x)-\phi_\delta^i(y)\right|^p\left|x-y\right|^{-n-ps}\,dx\,dy\\ \nonumber
 &\leq \frac{L^p\delta^{-p}}{\theta}\int_{\mathbb{R}^n}\int_{\mathbb{R}^n\cap\left\{\left|x-y\right|\leq \delta\right\}} \left|u_j(x)\right|^p\left|x-y\right|^{p-n-ps}\,dx\,dy\\
&\quad+\frac{2^p}{\theta}\int_{\mathbb{R}^n}\int_{\mathbb{R}^n\cap\left\{\left|x-y\right|>\delta\right\}} \left|u_j(x)\right|^p\left|x-y\right|^{-n-ps}\,dx\,dy\\
&\leq C\frac{(L^p\delta^{-p}+2^p)}{\theta}\int_{\mathbb{R}^n}\left|h(x)\right|^p\,dx\,dy<+\infty,\nonumber
\end{align}
with $C=C(n,s,\delta)>0$. So, by dominated convergence theorem
\begin{align}
\lim_{j\rightarrow+\infty}\int_{\mathbb{R}^{2n}}
|u_j(x)|^p |\phi_\delta^i(x)&-\phi_\delta^i(y)|^p |x-y|^{-n-ps}\,dx\,dy\notag \\
&=\int_{\mathbb{R}^{2n}}
\left|u(x)\right|^p\left|\phi_\delta^i(x)-\phi_\delta^i(y)\right|^p |x-y|^{-n-ps}\,dx\,dy.
\end{align}
Now, following the calculations as in equation \eqref{facl}, we get
\begin{align}
\int_{U\times V} |u(x)|^p |\phi_\delta^i(x)&-\phi_\delta^i(y)|^p |x-y|^{-n-ps}\,dx\,dy\notag\\
&\leq\frac{L^p}{\theta}\delta^{-p}\int_{U}\int_{V\cap\left\{\left|x-y\right|\leq \delta\right\}} \left|u(x)\right|^p \left|x-y\right|^{p-n-ps}\,dx\,dy\notag\\
&\quad+\frac{2^p}{\theta}\int_{U}\int_{V\cap\left\{\left|x-y\right|>\delta\right\}} \left|u(x)\right|^p\left|x-y\right|^{-n-ps}\,dx\,dy,
\end{align}
where $U$ and $V$ are two generic subsets of $\mathbb{R}^n$. Next, we claim that
\[
\int_{\mathbb{R}^{2n}}
\left|u(x)\right|^p\left|\phi_\delta^i(x)-\phi_\delta^i(y)\right|^p |x-y|^{-n-ps}\,dx\,dy\rightarrow 0,\;\text{as}\; \delta \rightarrow 0.\]
When $U=V=\mathbb{R}^n\setminus B(x_i,\delta)$  claim follows. When $U\times V=B(x_i,\delta)\times\mathbb{R}^n$ and $U\times V=\mathbb{R}^n\times B(x_i,\delta)$, we can use Proposition 2.1 of \cite{pmks} to prove the claim. Thus \\
\begin{equation}
\lim_{\delta\rightarrow 0}\int_{\mathbb{R}^{2n}}
\left|u(x)\right|^p\left|\phi_\delta^i(x)-\phi_\delta^i(y)\right|^p |x-y|^{-n-ps}\,dx\,dy=0.
\end{equation}
Hence
\begin{equation}\label{3.34}
M (\left\|u_j\right\|^{p}_{X_0})\int_{\mathbb{R}^{2n}} u_j(x)|(u_j(x)-u_j(y)|^{p-2}(u_j(x)-u_j(y))\big(\phi_\delta^i(x)-\phi_\delta^i(y)\big)|x-y|^{-n-ps}\,dx\,dy\rightarrow 0,
\end{equation}
as $\delta \rightarrow 0$ and $j\rightarrow \infty$.
Now, using H\"older's inequality
\begin{align}
\left|\int_{\mathbb{R}^n}\frac{u_j(x)-u_j(y)} {\left|x-y\right|^{n+ps}}\right|^p\notag
\leq &\,2^{p-1}\left[\left|u_j(y)\right|^p\left|\int_{\mathbb{R}^n\setminus\Omega}\frac{1}{\left|x-y\right|^{n+ps}}\right|^p+\left|
\int_{\Omega}\frac{u_j(x)-u_j(y)}{\left|x-y\right|^{n+ps}}dx\right|^p\right]\notag\\\label{ext1}
\leq &\,C_1\left|u_j(y)\right|^p+C_2\int_{\Omega}\left|u_j(x)-u_j(y)\right|^p|x-y|^{-n-ps}dx,
\end{align}
where $C_1=2^{p-1}|\int_{\mathbb R^n\setminus\Omega}\frac{dx}{|x-y|^{n+ps}}|^p \text{and}\; C_2=\frac{2^{p-1}}{\theta}.$
Now using equation \eqref{ext1} and \eqref{prkh1}, we get
\begin{align}\label{3.36}
\liminf_{j\rightarrow+\infty} \int_{\mathbb{R}^n} \phi_\delta^i(y)\nonumber
 &\int_\Omega\left|u_j(x)-u_j(y)\right|^p K(x-y)\,dx\,dy \notag\\ \nonumber
&\geq C_3\frac{1}{c(n,s)}\liminf_{j\rightarrow+\infty}\int_{\mathbb{R}^n} \phi_\delta^i(y) \nonumber \,c(n,s)\left|\int_{\mathbb{R}^n}\frac{u_j(x)-u_j(y)}{\left|x-y\right|^{n+ps}}\,dx\right|^p dy \notag\\
&\quad-C_4\liminf_{j\rightarrow+\infty}\int_{\mathbb{R}^n} \phi_\delta^i(y)\left|u_j(y)\right|^p dy\notag\\
&\geq C_3\frac{1}{c(n,s)}\int_{\mathbb{R}^n} \phi_\delta^i(y)d\mu-C_4\int_{B(x_i,\delta)} \left|u(y)\right|^p dy,
\end{align}
where $C_3=\frac{1}{C_2}$ and $C_4=\frac{C_1}{C_2}$.
Moreover,
\begin{equation}\label{3.38}
\int_{B(x_i,\delta)} f(x)|u_j(x)|^{q-2}u_j(x)\phi_\delta^i(x)dx\rightarrow\int_{B(x_i,\delta)} f(x)|u(x)|^{q-2}u(x)\phi_\delta^i(x)dx,
\end{equation}
as $j\rightarrow+\infty$. We also observe that the integral goes to 0 as $\delta\rightarrow0$.
So, using  equations \eqref{3.34}, \eqref{3.36}, \eqref{3.38} and \eqref{prkh1}) in  equation \eqref{3.27}, we get
\begin{equation*}
\begin{alignedat}2
&\int_\Omega \phi_\delta^i(x)d\nu+\lambda\int_{B(x_i,\delta)} f(x)|u(x)|^{q-2}u(x)\phi_\delta^i(x)dx\\
&\qquad\qquad\quad\geq M(\alpha^p)C\left(\int_\Omega\phi_\delta^i(y)d\mu-\int_{B(x_i,\delta)}\left|u(y)\right|^p dy\right)+o_\delta(1).
\end{alignedat}
\end{equation*}
Now, by taking $\delta\rightarrow 0$, we conclude that $\nu_i\geq M(\alpha^p)C\mu_i\geq m_0 C\mu_i$. Then by  \eqref{prkh2}, we get
\begin{equation}\label{4.6}
\nu_i\geq\displaystyle\frac{(m_0 C)^{n/ps}}{S^{(n-ps)/ps}},
\end{equation}
for any $i\in J$, where $C=\frac{C_3}{c(n,s)}$, independent of $\lambda$. We will prove that equation \eqref{4.6} is not possible.

From the assumption in the Lemma, we get
{
\begin{align*}
c&>\left(\frac{p_s^*-2p}{2pp_s^*}\right)\int_\Omega|u_j|^{p_s^*}+\frac{a}{2p}\|u_j\|_{X_0}^p-\lambda\left(\frac{2p-q}{2pq}\right)\int_\Omega f(x)|u_j|^qdx\\
&>\left(\frac{p_s^*-2p}{2pp_s^*}\right)\int_\Omega|u_j|^{p_s^*}+\frac{a}{2p}\|u_j\|_{X_0}^{rq}-\lambda\left(\frac{2p-q}{2pq}\right)l^{\frac{p-q}{p}}S^{\frac{1}{r}}(\|u_j\|_{X_0}^{rq})^{\frac{1}{r}}\;\;
\end{align*}
}
{for some $r=\frac{p}{q}>1.$} Now consider $h(t)=\frac{a}{2p}t^r-\lambda\left(\frac{2p-q}{2pq}\right)l^{\frac{p-q}{p}}S^{\frac{1}{r}}t$,\; which has minimizer at $t_0=\left[\frac{\lambda(2p-q)l^{\frac{p-q}{p}}S^\frac{1}{r}}{ar}\right]^{\frac{1}{r-1}}$. Hence
{
\begin{align*}
c&>\left(\frac{p_s^*-2p}{2pp_s^*}\right)\int_\Omega|u_j|^{p_s^*}-\left(\frac{1}{r^{\frac{1}{r-1}}}-\frac{1}{r^{\frac{r}{r-1}}}\right)\frac{(\lambda(2p-q)l^{\frac{p-q}{p}}S^{\frac{1}{r}})^{\frac{r}{r-1}}}{2pqa^{\frac{1}{r-1}}}\\
&>\left(\frac{p_s^*-2p}{2pp_s^*}\right)\int_\Omega\phi_\delta^i|u_j|^{p_s^*}-C\lambda^{\frac{p}{p-q}}
\end{align*}}
As $j\rightarrow \infty\; \text{and}\; \delta\rightarrow 0$, we have
\begin{equation}
c>\left(\frac{p_s^*-2p}{2pp_s^*}\right)\displaystyle\frac{(m_0 C)^{n/ps}}{S^{(n-ps)/ps}}-C\lambda^{\frac{p}{p-q}}
\end{equation}
where $C=\left(\frac{1}{r^{\frac{1}{r-1}}}-\frac{1}{r^{\frac{r}{r-1}}}\right)\frac{((2p-q)l^{\frac{p-q}{p}}S^{\frac{1}{r}})^{\frac{r}{r-1}}}{2pqa^{\frac{1}{r-1}}}$
, which is a contradiction. Therefore $\nu_i=0, \;\forall\; i\in J $. Hence $J$ is empty. Which implies $u_j\rightarrow u$ in $L^{p_s^*}(\Omega)$. So, by equation $\eqref{pscn1}$ taking $\phi = u_j $ and using dominated convergence theorem,
\begin{equation}\label{3.41}
\lim_{j\rightarrow+\infty}M(\left\|u_j\right\|_{X_0}^{p})\left\|u_j\right\|_{X_0}^{p}=\lambda\int_\Omega f(x)|u(x)|^qdx+\int_\Omega\left|u(x)\right|^{p_s^*}dx.
\end{equation}
Now, we take $\phi=u$ in equation $\eqref{pscn1}$  and recalling that $M(\left\|u_j\right\|^{p}_{X_0})\rightarrow M(\alpha^p)$, we get
\begin{equation}\label{4.10}
M(\alpha^p)\left\|u_j\right\|_{X_0}^{p}=\lambda\int_\Omega f(x)u(x)\,dx-\int_\Omega \left|u(x)\right|^{p_s^*}dx.\
\end{equation}
So, combining  equations \eqref{3.41} and \eqref{4.10}, we get
\[M(\left\|u_j\right\|_{X_0}^{p})\left\|u_j\right\|_{X_0}^{p}\rightarrow M(\alpha^p)\left\|u\right\|_{X_0}^{p},\qquad\text{as}\,\,j\rightarrow+\infty.
\]
So, using this result, we have
\begin{align}
M(\left\|u_j\right\|_{X_0}^{p})(\left\|u_j\right\|_{X_0}^{p}-\left\|u\right\|_{X_0}^{p})= & M(\left\|u_j\right\|_{X_0}^{p})\left\|u_j\right\|_{X_0}^{p}
-M(\alpha^p)\left\|u\right\|_{X_0}^{p}\notag \\
&\quad \quad -M(\left\|u_j\right\|_{X_0}^{p})\left\|u\right\|_{X_0}^{p}+M(\alpha^p)\left\|u\right\|_{X_0}^{p}, \end{align}
which leads to
\begin{equation}
M(\left\|u_j\right\|_{X_0}^{p})(\left\|u_j\right\|_{X_0}^{p}-\left\|u\right\|_{X_0}^{p})\rightarrow 0.
\end{equation}
Also
\begin{equation}
m_0(\left\|u_j\right\|_{X_0}^{p}-\left\|u\right\|_{X_0}^{p})\leq M(\left\|u_j\right\|_{X_0}^{p})(\left\|u_j\right\|_{X_0}^{p}-\left\|u\right\|_{X_0}^{p}),
\end{equation}
which implies $\left\|u_j\right\|_{X_0}^{p}\rightarrow \left\|u\right\|_{X_0}^{p}$ and the claim is proved. Hence $u_j\rightarrow u$ strongly in $X_0$.
\end{proof}
\noindent \textbf{Proof of Theorem \ref{th2}:} Assume $\lambda_0^1=\max\{\lambda: \theta_\lambda \leq \left(\frac{p_s^*-2p}{2pp_s^*}\right)\displaystyle\frac{(m_0 C)^{n/ps}}{S^{(n-ps)/ps}}-C\lambda^{\frac{p}{p-q}} \}$.
and $\lambda_{00}=\min\{\lambda_0, \lambda_0^1\}$. Now by Lemma \ref{le44} and Ekeland variational principle,  we get a minimizing sequence $\{u_k\}$ satisfying
\begin{equation}
\mathcal{J}_\lambda(u_k)=\theta_\lambda +o(1) \;\text{and}\; \mathcal{J}'_\lambda(u_k)=o(1).
\end{equation}
then $\{u_k\}$ is bounded in $X_0$. Hence, upto subsequence,  $u_k\rightharpoonup u_\lambda$ in $X_0$, $u_k\rightarrow u_\lambda$ in $L^q \;\forall\; q \in [1, p_s^*)$ and $u_k(x)\rightarrow u_\lambda(x)$ a.e. in $\Omega$. Now by Lemma \ref{cmcr}, $u_k\rightarrow u_\lambda$ in $X_0$ for $\lambda \in (0, \lambda_{00})$. So by Lemma \ref{lmcp}, $u_\lambda$ is a solution of problem $P_\lambda$ for $r=p_s^*$. Hence $u_\lambda \in \mathcal{N}_\lambda.$ Now using equation \eqref{psr}, we get $u_\lambda\in \mathcal{N}_\lambda^+.$

\end{document}